\documentclass[draft,envcountsect]{svjour3}       
\smartqed  
\newcommand{\supp}{{\bf supp}}

\usepackage{amsmath,amssymb}
\usepackage{hyperref}
\usepackage{natbib}
\usepackage{xcolor}
\usepackage{enumerate}
\journalname{Preprint}
\begin{document}

\title{Evolutionary Stability Against Multiple Mutations
}
\author{Anirban Ghatak \and K.S.Mallikarjuna Rao  \and A.J. Shaiju 
}
\authorrunning{Ghatak, Rao and Shaiju} 
\institute{Anirban Ghatak \at
              Industrial Engineering and Operations Research,
              Indian Institute of Technology Bombay,
              Mumbai 400076, India. \\
              \email{aghatak@iitb.ac.in}           
           \and
              K. S. Mallikarjuna Rao \at
              Industrial Engineering and Operations Research,
              Indian Institute of Technology Bombay,
              Mumbai 400076, India. \\
              \email{mallik.rao@iitb.ac.in}
            \and
              A. J. Shaiju \at
              Department of Mathematics,
              Indian Institute of Technology Madras,
              Chennai 600036, India. \\
              \email{ajshaiju@iitm.ac.in}
}

\maketitle

\begin{abstract}
It is known (see e.g.\ \citet{Wei-95}) that ESS is not robust against 
multiple mutations. In this article, we
introduce robustness against multiple mutations and
study some equivalent formulations and  consequences. 
\end{abstract}

\section{Introduction}
The key concept in evolutionary game theory is ESS introduced by \citet{MP-73}
and early developments and applications to evolutionary biology are reported
in \citet{May-82}. Some of the references to modern developments include
\citet{Cressman}, \citet{HS-98}, \citet{Wei-95}. ESS deals with the situation when
there is only one rare mutation that can influence the population.

In practical scenarios, an incumbent strategy may be subjected to invasions by
several mutations. As an example, we can consider bird nesting. During the season,
birds search for a good location. To get the best location, they may need to compete
with several others. In game theoretic terminology, this corresponds to the invasion
of multiple mutations. Thus it is desirable to study the influence of multiple mutations.
To the best of our knowledge, we are not aware
of any studies which have dealt with the case of multiple mutations. 

In \citet{Wei-95} (and \citet{VincentBrown}),
it is noted that ESS is, in general,  not robust against multiple mutations.
One possible way to approach this issue may be to model this situation as a
multi-player game
and study the corresponding  ESS.  While there have been works dealing
with ESS and multi-player games (see e.g.\ \citet{BroomCanningsVickers}),
we are not aware of
any on the connection of this theory
with multiple mutations.
In this work, we take alternative approach, where we extend the notion of ESS to
take care of  the multiple mutations. We provide some interesting consequences.
Our approach  leads to useful refinement of Nash equilibrium.

Our article is structured as follows. After the introductory section, we introduce
evolutionary stability against multiple mutations in Section 2. We show that
evolutionary stability against multiple mutations is equivalent to evolutionary stability
against two mutations in a special case. In Section 3, we provide an equivalent formulation.
Using the ideas in this equivalent formulation, we show that evolutionary stability 
against multiple mutations is equivalent to evolutionary stability against two mutations.
Section 4 introduces the concept of local dominance  and 
discusses its connections with evolutionary stability. It also draws differences 
with strict Nash equilibrium.  In Section 5, we establish the
fact that evolutionarily stable strategy against multiple mutations is necessarily a 
pure strategy, a property shared by strict Nash equilibrium. One consequence of
this fact is the existence of uniform invasion barrier. We characterize the 
evolutionarily stable strategies against multiple mutations in $2 \times 2$ games.
We conclude our article with some comments and 
directions for further research in Section 6.

\section{Evolutionary Stability}
We consider symmetric games with payoff function $u: \Delta \times
\Delta \to \mathbb{R}$, where $\Delta$ is probability simplex in
$\mathbb{R}^k$ and $u$ is given by the affine function
\[
u(p, q) = \sum_{i,j=1}^k p_i q_j u(e^i, e^j).
\]
Here $e^1=(1,0,0,\cdots,0),\cdots,
 e^k=(0,\cdots,0,1) \in{\mathbb{R}}^k $  denote
  the
 pure strategies of the players.
 We, first recall the definition of
 an evolutionarily stable strategy (ESS for short).
\begin{definition}
A  strategy $p \in \Delta$ is called ESS, if for any mutant strategy
$r \neq p$, there is an invasion barrier  $\epsilon(r) \in (0,1)$
such that
 \begin{equation} \label{ess1}
u(p,\epsilon r + (1-\epsilon)p) >
u(r,\epsilon r + (1-\epsilon)p)
\mbox{ for all } 0<\epsilon \leq \epsilon(r).
\end{equation} \label{ess-defn}
\end{definition}

We gather some  notations that we use in due course:
\begin{eqnarray*}
BR(p) &=& \{ q \in \Delta ~:~
u(q,p)\geq u(r,p)~~\forall r \in
\Delta  \},\\
\Delta^{NE}&=&
\{ p\in \Delta  ~:~ p\in BR(p)  \}.
\end{eqnarray*}

By definition, an ESS is robust against
 any single mutation $r$ appearing in small
  proportions.
 A natural question that  arises  is that whether
an ESS is robust against multiple mutations.
 It is known  (see e.g.
\citet{Wei-95}) that an
 ESS may not be
robust against multiple mutations.
 We now provide an example to
illustrate this fact.
\begin{example}
Consider the $2 \times 2$ symmetric game with
 fitness (or, payoff)
matrix $U= \begin{pmatrix} -1 & 0 \\ 0 & -1  \end{pmatrix}$. The
unique ESS of this game is $p=(\frac{1}{2},\frac{1}{2})$.

Consider $r^1=(\frac{1}{4},\frac{3}{4})$ and
$r^2=(\frac{3}{4},\frac{1}{4})$. Now, for any $0<\epsilon
<\frac{1}{2}$,
\[
u(p,\epsilon r^1 + \epsilon r^2 + (1-2\epsilon)p) =  -\frac{1}{2}.
\]
But
\[
u(r_1,\epsilon r^1 + \epsilon r^2 + (1-2\epsilon)p) = - \left[
\epsilon(\frac{10}{16}+\frac{6}{16} )   +(1-2\epsilon) \frac{1}{2}
\right] =-\frac{1}{2}.
\]
Thus $p$ is not robust against simultaneous mutations $r^1,r^2$,
whenever they appear in equal propositions.
\end{example}

The above example makes it clear that
the Definition \ref{ess-defn} is inadequate to
capture the robustness or evolutionary stability
against multiple mutations. This motivates the
following definition.

\begin{definition} \label{mess-defn}
 Let $m$ be a positive integer.
A  strategy $p \in \Delta$ is said to be evolutionarily
stable (or robust) against  `$m$' mutations if, for every  
$r^1,\cdots,r^m \neq p$, there exists
$\bar{\epsilon}=\bar \epsilon(r^1,\cdots,r^m) \in (0,1)$
such that
\begin{multline*}
u(p,\epsilon_1 r^1+\cdots + \epsilon_m r^m + (1-\epsilon_1-\cdots
 -\epsilon_j)p) \\
{} > \max_{1\leq i \leq m} u(r^i,\epsilon_1 r^1+\cdots + \epsilon_j
r^j +
 (1-\epsilon_1-\cdots  -\epsilon_j)p),
\end{multline*}
for all $\epsilon_1, \ldots,\epsilon_j \in (0,\bar{\epsilon}]$.
\end{definition}
\begin{remark}
Clearly if $m=1$, then the above definition coincides with the definition of ESS.
\end{remark}

\begin{definition}
If $\bar \epsilon$ in Definition \ref{mess-defn}  can be chosen independent of
$(r^1,r^2,\cdots,r^m)$, then we refer to $\bar \epsilon$ as a 
{\it uniform invasion barrier} for $p$ corresponding to $m$ mutations. 
\end{definition}

\begin{remark}
The uniform invasion barrier $\bar{\epsilon}$, in general, depends on $m$. However,
we can choose a bound on the total fraction i.e., 
$\epsilon_1 + \epsilon_2 + \cdots + \epsilon_m$ of the $m$ mutations can be taken 
to be independent of $m$. We come back to this point later. 
\end{remark}

We now provide  a surprising result albeit with a very simple proof.

\begin{theorem} \label{thm-2m}
Let $m>2$ and $p\in \Delta$. If the strategy $p$  is evolutionarily stable against 
two mutations with uniform invasion barrier, then it is  evolutionarily stable against 
$m$ mutations with uniform invasion barrier. 
\end{theorem}

\begin{proof}
Let $p$ be evolutionarily stable against two mutations with uniform invasion
barrier $\bar{\epsilon}$. 
Let $r^1, r^2, \cdots,  r^m$ be  $m$  strategies different from $p$. 

Let $\epsilon_i \in (0, \frac{\bar{\epsilon}}{m}]$, $i=1,2, \cdots, m$. Consider
\begin{multline*}
w = \epsilon_1 r^1 + \epsilon_2 r^ 2 + \cdots + \epsilon_m r^m  + (1 - \epsilon_1 -
\epsilon_2 -\cdots- \epsilon_m) p  \\
= (\epsilon_1 +\cdots + \epsilon_{m-1}) s +\epsilon_m r^m + 
(1 - \epsilon_1 - \epsilon_2 -  \cdots - \epsilon_m) p,
\end{multline*}
where
\[
s = \frac{\epsilon_1}{\epsilon_1 + \cdots +  \epsilon_{m-1}}   r^1 + 
\cdots +  \frac{\epsilon_{m-1}}{\epsilon_1   + \cdots +  \epsilon_{m-1}}  r^{m-1} .
\]
Obviously $(\epsilon_1 +\cdots + \epsilon_{m-1}) , \epsilon_m \leq \bar{\epsilon}$. 
Now considering $s$ and $r^m$ as mutations, we  have
\[
u (p, w) > \max \{ u( s, w), u( r^m, w) \},
\]
since $p$ is robust against two mutations with 
uniform invasion barrier $\bar{\epsilon}$. Thus
\[
u (p, w) > u( r^m, w).
\]
Instead of $r^m$, we can use any $r^j$, $j =1, 2, \cdots, m-1$ in 
the above analysis which leads to
\[
u (p, w) > u( r^j, w)
\]
for every $j = 1, 2, \cdots, m$. Thus $p$ is evolutionarily stable 
against $m$ mutations and  $\tfrac{\bar{\epsilon}}{m}$ is  uniform
invasion barrier corresponding to $m$ mutations. 
\qed
\end{proof}

\begin{remark}
The theorem assumes the existence of uniform invasion
barrier. This is, indeed, the case as we show later.
\end{remark}

%
%
%
%
%

\section{An Equivalent Formulation}

In this section, we provide an  equivalent formulation for the
evolutionary stability against two mutations.  

\begin{theorem}\label{charmulti}
For $p\in \Delta$, the following are equivalent:
\begin{enumerate}[{\bf (a)}]
\item  $p$ is robust against  two  mutations;

\item $p\in \Delta^{NE}$, and,
 for every $q \in BR(p) \setminus \{ p\}$ and
  $r\in \Delta$,
\[
  u(p,q)>u(q,q)  \quad \mbox{ and } \quad
u(p,r) \geq u(q,r) .
\]
\end{enumerate}
\end{theorem}
\begin{proof}
We start with (a) $\Rightarrow$ (b). Assume that $p$ is robust
against two mutations.  In particular $p$ is an ESS. Let $q \in
BR(p) \setminus \{ p\}$ and $r\in \Delta$. For small enough
$\epsilon_1,\epsilon_2>0 $, we must have
\[
u(p,(1-\epsilon_1-\epsilon_2)p + \epsilon_1q +   \epsilon_2r) >
u(q,(1-\epsilon_1-\epsilon_2)p +  \epsilon_1q+\epsilon_2r).
\]
Rearranging the terms, we get
\[
\epsilon_1 \{  u(p, q) - u(q, q) \} +  \epsilon_2 \{ u(p, r) - u(q, r) \}
+ (1 - \epsilon_1 - \epsilon_2) \{ u(p, p) - u(q, p) \} > 0.
\]
Since $q \in BR(p)$, the third term is zero and
 hence,
 for small enough $\epsilon_1,\epsilon_2>0$,
  we have
\[
\epsilon_1[u(p,q)-u(q,q)]+
\epsilon_2[u(p,r)-u(q,r)] >0.
\]
Since $p$ is an ESS and $q   \in BR(p)\setminus
 \{ p \}$, we have
$u(p,q)>u(q,q) $. From this and the
 above inequality, it follows that
  $u(p,r) \geq u(q,r)$.

We now show that (b) $\Rightarrow$ (a). Assume (b). 
Let the mutations $r^1, r^2$ appear in  proportions $\epsilon_1,\epsilon_2$
respectively. For $i=1,2$, let
\[
h_i(\epsilon_1,\epsilon_2) :=
u(p,\epsilon_1 r^1+\epsilon_2 r^2 + (1-\epsilon_1-\epsilon_2)p) -
u(r_i,\epsilon_1 r^1+\epsilon_2 r^2+ (1-\epsilon_1- \epsilon_2)p)
\]
We need to show that for $\epsilon_1,\epsilon_2$ small enough,
$h_i(\epsilon_1, \epsilon_2) > 0$ for each $i=1,2$. Note that
\begin{multline}\label{h-defn}
h_i(\epsilon_1, \epsilon_2) =  \epsilon_1[u(p,r^1)-u(r_i,r^1)]  \\
{ } + \epsilon_2 [u(p,r^2)-u(r^i,r^2)] + (1-\epsilon_1-\epsilon_2)[u(p,p) - u(r^i,p)].
\end{multline}
Fix $i$. If $r^i \in BR(p)$, then the third term on the R.H.S. of \eqref{h-defn} is zero.
By hypothesis, $u(r^i, r^i) < u(p, r^i)$ and $u(r^i, r^j) \leq u(p, r^j)$, for $j \neq i$.
Therefore, for $\epsilon_1,\epsilon_2>0$,
$h_i(\epsilon_1,\epsilon_2)>0$ whenever $r^i \in BR(p)$.
 
Now let $r^i \not \in BR(p)$. Then $u( p, p) - u(r^i, p) > 0$.
Hence for sufficiently small $\epsilon_1 $ and $ \epsilon_2$,
we must have  $h(\epsilon_1,\epsilon_2)>0$. Thus $p$ is 
robust against two mutations.
\qed
\end{proof}

\begin{remark}
The above characterization suggests the following interpretation of
evolutionary stability against multiple mutations: An
ESS is robust against multiple mutations if
and only if it dominates all strategies that are best responses to it.
\end{remark}

A careful observation of  the  proof of Theorem \ref{charmulti} shows
that evolutionary stability against 2 mutations is equivalent to evolutionary 
stability against any  $m$ mutations, $m \geq 2$. We now prove this
equivalence. Note that in the previous section, 
we showed this result when there is uniform invasion barrier. 

\begin{theorem}\label{charmultithm2}
A strategy is evolutionarily stable against two mutations if and only if it is evolutionarily
stable against  $m$ mutations, where $m > 2$.
\end{theorem}

\begin{proof}
We will only show that evolutionary stability against two mutations implies the evolutionary
stability against $m$ mutations, the other part being trivial.

Let $p$ be evolutionarily stable against two mutations. 
Let  $r^1, r^2, \cdots, r^m$ be $m$ mutations that appear
with proportions $\epsilon_1, \epsilon_2, \cdots, \epsilon_m$, respectively.
For $i = 1, 2, \cdots, m$, let
\begin{multline*}
h_i(\epsilon_1,\epsilon_2, \cdots, \epsilon_m) :=
u(p,\epsilon_1 r^1+\epsilon_2 r^2 +  \cdots + \epsilon_m r^m + 
(1-\epsilon_1-\epsilon_2 - \cdots - \epsilon_m ) p )  \\
{} - u(r_i,\epsilon_1 r^1+\epsilon_2 r^2+ \cdots + \epsilon_m r^m +  
(1-\epsilon_1- \epsilon_2 - \cdots - \epsilon_m )p)
\end{multline*}
We need to show that for $\epsilon_1,\epsilon_2, \cdots, \epsilon_m$ small enough,
$h_i(\epsilon_1, \epsilon_2, \cdots,  \epsilon_m) > 0$ for each $i=1,2, \cdots, m$. 
Note that
\begin{multline}\label{h-defn2}
h_i(\epsilon_1, \epsilon_2, \cdots, \epsilon_m) =  \epsilon_1[u(p,r^1)-u(r^i,r^1)]   
+ \epsilon_2 [u(p,r^2)-u(r^i,r^2)]  \\ 
{ } + \cdots   + \epsilon_m [u(p, r^m) - u(r^i, r^m)] \\ 
{ }  +   (1-\epsilon_1 - \epsilon_2 - \cdots - \epsilon_m ) [u(p,p) - u(r^i,p)].
\end{multline}
Fix $i$. If $r^i \in BR(p)$,  then  $u(r^i, p) - u(p, p) = 0$. From Theorem \ref{charmulti},
we have
\[
u(r^i, r^i) < u(p , r^i) \text{ and } u(r^i, r^j) \leq u(p, r^j)
\]
for all $j \neq i$. As a result, we have $h_i(\epsilon_1,\epsilon_2, \cdots, \epsilon_m) > 0$
for $\epsilon_1,\epsilon_2, \cdots, \epsilon_m > 0$, whenever $r^i \in BR(p)$.
 
Now let $r^i \not \in BR(p)$. Then $u( p, p) - u(r^i, p) > 0$.
Thus for sufficiently small $\epsilon_1,  \epsilon_2, \cdots, \epsilon_m > 0$,
we must have  $h(\epsilon_1,\epsilon_2, \cdots, \epsilon_m) > 0$.  And 
hence $p$ is evolutionarily stable against $m$ mutations.
\qed
\end{proof}

\begin{remark}
As a result of the Theorem \ref{charmultithm2},  if a strategy is evolutionarily 
stable 
against $m\geq 2$ mutations, we refer to it simply as evolutionarily stable
against multiple mutations, by suppressing the number $m$.
\end{remark}

\section{Local Dominance}

In  evolutionary game theory, ESS is characterized by means of two
notions: uniform invasion barrier and local superiority.  Uniform invasion 
barrier is already introduced. Local superiority of a mixed strategy
$p$ implies that $u(p, q) > u(q, q)$ for every $q\neq p$ in a neighborhood of
$p$. The interpretation of this notion is as follows: $p$ is ESS if and only if
in a neighborhood of $p$, there can not be any other symmetric Nash equilibrium
other than $p$. We now introduce the corresponding generalization of local superiority to the case of 
multiple mutations.

\begin{definition}[Local Dominance]
A  strategy $p\in \Delta$ is said to be locally dominant if
there is a neighborhood $U$ of $p$ such  that
 $u(p,r) \geq u(s,r)$ and $u(p,r)>u(r,r)$
for every $s, r \in U
 \setminus \{ p \}$.
\end{definition}

\begin{remark}
Note that if $p$ is  locally dominant, then we can easily show
that
\[
u(p,r) \geq u(s,r) \text{ and } u(p,r)>u(r,r)
\]
for every $s \in \Delta$ and $ r \in U \setminus \{ p \}$, 
where $U$ is the neighborhood as in the
definition of local dominance.
\end{remark}

We now show that evolutionary stability against multiple mutations and 
local dominance are equivalent. 

\begin{theorem}
A strategy $p$ is evolutionarily stable  against multiple mutations if and only if it is 
locally dominant.
\end{theorem}

\begin{proof}
Assume that $p$ is locally dominant. By definition, $p$ is an ESS.
Let $q  \in BR(p)$ and $q, r\neq p$. To show that $p$ is robust against
multiple mutations,  it suffices to show that $u(p,r)\geq u(q,r)$.

Note that $r^\epsilon=\epsilon r+(1-\epsilon)p$
is close to $p$ whenever $\epsilon>0$ is small enough. Since $p$ is weak locally
dominant, for $\epsilon>0$ small enough,
we must have
\begin{eqnarray*}
0 &\leq& u(p,r^\epsilon)-u(q,r^\epsilon)
= \epsilon [u(p,r)-(q,r)].
\end{eqnarray*}
This implies that $u(p,r)\geq u(q,r)$.

Now assume that $p$ is robust against multiple
mutations. Let $s\neq p$. We first show that
there exists a neighborhood $V=V(s)$ of $p$
such that
\begin{equation} \label{wkld-ineq}
f(r):= u(p,r)- u(s,r) \geq 0
\end{equation}
for all $r\in V
\setminus \{ p \}$.

Now
\[
f(e^i_\epsilon)= \epsilon[u(p,e^i)-u(s,e^i)]+
(1-\epsilon)[u(p,p)-u(s,p)],
\]
where $e^i_\epsilon = \epsilon e^i+(1-\epsilon)p$.\\
If $s \in BR(p)$, then, by hypothesis,
$f(e^i_\epsilon)\geq 0$ for every $0\leq \epsilon \leq 1$.
If $s\notin BR(p)$,
then clearly there exists $\bar \epsilon_i (s)
\in (0,1)$
such that
$f(e^i_\epsilon)> 0$
for $0\leq \epsilon <\bar \epsilon_i(s)$.

Thus  $f(r) \geq 0$ when $r\in L$;
\[
L=\{ w \in \Delta ~:~ w=\epsilon e^i+ (1-\epsilon)p ~ \mbox{ for some }
1 \leq i \leq k,~
0 \leq \epsilon< \min_{1\leq i \leq k}
\bar \epsilon_i(s) \}.\\
\]
This clearly implies that $f(r) \geq 0$  for every $r$
in the convex hull $V=V(s)$ (which is also a neighborhood of $p$)
of $L$.  Therefore $u(p,r) \geq u(s,r)$ for every $s$ and
$r \in U :=\cap_{i=1}^kV(e^{i})$. This implies that $p$ is  locally dominant.
\qed
\end{proof}

In the following proposition we show that the inequality in the local dominance
is strict for all $s$ whose support has a non-empty intersection with the support
of $p$.

\begin{theorem}
Let $p$ be robust against multiple mutations. Then there is a neighborhood
$U$ such that $u(p, r) > u(s, r)$ for all $r \in U$ and $s \in U$ such that
$\supp(s) \cap \supp(p) = \emptyset$.
\end{theorem}

\begin{proof}
Let $p$ be evolutionarily stable against multiple mutations and
let
\[
C = \{ p \in \Delta : p_i = 0 \text{ for some } i \in \supp(p) \}.
\]
Clearly $C$ is compact and $p \in C$.  We can choose $\bar{\epsilon} > 0$
such that
\[
u(p, \epsilon_1 r + \epsilon_2 s + (1-\epsilon_1 + \epsilon_2) p)
> u(r,  \epsilon_1 r + \epsilon_2 s + (1-\epsilon_1 - \epsilon_2) p)
\]
for all $r, s \in C$ and $0 < \epsilon_1, \epsilon_2 \leq \bar{\epsilon}$.
Hence
\[
u(p, \epsilon_1 r + \epsilon_2 s + (1-\epsilon_1 + \epsilon_2) p)
>  u(\alpha_1 r + \alpha_2 s + (1-\alpha_1 + \alpha_2) p,
\epsilon_1 r + \epsilon_2 s + (1-\epsilon_1 + \epsilon_2) p)
\]
for every $0 < \alpha_1, \alpha_2, \epsilon_1, \epsilon_2 < \bar{\epsilon}$.
If we choose $U = B(p; \bar{\epsilon})$, then from the above we have
\[
u(p, r) > u(s, r)
\]
for all $r, s \neq p \in U$.
\qed
\end{proof}

We now make a definition.

\begin{definition}[Strict Local Dominance]
A  strategy $p\in \Delta$ is said to be strictly locally dominant if
there is a neighborhood $U$ of $p$ such  that
$u(p,r) >  u(s,r)$  for every $s, r \in U
 \setminus \{ p \}$.
\end{definition}

A strict Nash equilibrium is always strictly locally dominant. We
may think that the other way is also correct. However it is not
the case as the following example shows.

\begin{example}
Consider the $2 \times 2$ symmetric game with fitness
matrix
$$U=\left(  \begin{array}{rr} -1 & 0 \\ 0 & 0   \end{array} \right). $$
Clearly
$BR(e^2)=\Delta$, and hence it is not a strict symmetric Nash equilibrium.
But $e^2$ is an ESS, since for any $q\neq e^2$,
\[
u(q,q)=-q_1^2<0=u(e^2,q).
\]
Furthermore, for $q,r \neq e^2$,
\[
u(q,r)=-q_1r_1 <0 =u(e^2,r).
\]
Therefore, by Theorem \ref{charmulti}, $e^2$ is robust against
$m$ mutations, for any $m\geq 1$.
\end{example}

\section{Pure Strategies and Uniform Invasion Barrier}
An ESS can be mixed. On contrary, evolutionarily stable strategy against
multiple mutations is always pure. We now prove this fact.
Note that a strict Nash equilibrium is also necessarily
pure.

\begin{theorem}
An evolutionarily stable strategy against multiple mutations is necessarily a pure
strategy.
\end{theorem}

\begin{proof}
Let $p$ be evolutionarily stable against
multiple  mutations. If possible, let $p$ be not a pure strategy. Let
$p = (p_1, p_2, \cdots, p_k)$.
Let $\bar \epsilon = \bar \epsilon (e^1,e^2, \cdots, e^k)$ be 
the invasion barrier corresponding
to all  the $k$ pure mutations.
Let $r = \alpha_1 e^1 + \alpha_2 e^2 + \cdots + \alpha_k e^k + 
(1 - \alpha_1 - \alpha_2 - \cdots - \alpha_k) p$, where $0
< \alpha_1, \alpha_2, \cdots, \alpha_k < \bar{\epsilon}$.
 Then, we have
\begin{equation}\label{pure1}
\sum_{i=1}^k p_i u(e^i, r) = u(p, r)
> \max\{u(e^1, r), u(e^2, r), \cdots, u(e^k, r) \},
\end{equation}
which is a contradiction. Thus $p$ must be pure.
\qed \end{proof}

So far we have not answered the question of existence of uniform invasion 
barrier. We now answer this question. 

\begin{theorem}\label{unifinvthm}
If $p$ is robust against multiple mutations, then it has uniform invasion barrier.
\end{theorem}

\begin{proof}

Let $p$ be robust against multiple mutations. Then $p$ is necessarily pure. 
Without loss of generality, let us assume that $p = e^k$.  

Let $\bar{\epsilon}$ be the invasion 
barrier corresponding to  the pure strategies $e^1, \cdots, e^{k-1}$. We show that $\frac{\bar{\epsilon}}{m}$ is
invasion barrier for any $m$ mutations with $m \geq k - 1$.

Let $r^1, r^2, \cdots, r^m$ be $m$ mutations with proportions $\epsilon_1, \epsilon_2, 
\cdots, \epsilon_m$  respectively. Choose $\alpha_i^j$, 
$i=1,2, \cdots, m$, $j=1,2, \cdots, k$ such that  $r^i =  \alpha_i^1 e^1 + 
\alpha_i^2 e^2 + \cdots +  \alpha_i^k e^k$. Consder 
\begin{multline*}
w = \epsilon_1 r^1 + \epsilon_2 r^2 + \cdots + \epsilon_m r^m  - (1 - \epsilon_1 
- \epsilon_2 - \cdots - \epsilon_m ) p \\
=  \beta_1 e^1 + \beta_2 e^2  + \cdots + \beta_k e^k + (1 - \beta_1 - \beta_2 - \cdots 
- \beta_k ) p  \\
= \beta_1 e^1 + \beta_2 e^2  + \cdots + \beta_{k-1} e^{k-1} + (1 - \beta_1 - \beta_2 - \cdots 
- \beta_{k-1} ) p
\end{multline*}
where
\[
\beta_i = \epsilon_1 \alpha_1^i + \epsilon_2 \alpha_2^i + \cdots + \epsilon_m \alpha_m^i 
\text{ and } i = 1, 2, \cdots, m.
\]
If we choose $\epsilon_1, \epsilon_2,  \cdots, \epsilon_m \leq \frac{\bar{\epsilon}}{m}$,
then from the definition of evolutionary stability we have,
\[
u(p, w) > u(e^j, w), j=1,2, \cdots, k-1.
\]
Thus for any  $i$, $i=1,2, \cdots, m$, we have
\[
u(p, w) = \sum_{j=1}^{k} \alpha_i^j u(p, w) > 
 \sum_{j=1}^{k} \alpha_i^j u(e^j, w) = u(r^i, w).
\]
Here we have used the above $k-1$ inequalities together with the fact that 
$p = e^k$. Thus $p$ is evolutionarily stable against $m$ mutations with 
$\frac{\bar{\epsilon}}{m}$ as the uniform invasion barrier.

Note that any invasion barrier corresponding to $m$ mutations is also invasion
barrier corresponing to $n$ mutations, where $n < m$. This completes the proof
the thoerem.
\qed
\end{proof}

\begin{remark}
As a consequence of the proof, we note that the bound on the total fraction 
of the $m$ mutations $\epsilon_1 + \epsilon_2 + \cdots + \epsilon_m$
can be chosen to be $\bar{\epsilon}$, which is independent of $m$.
\end{remark}

A careful observation of the proof of the Theorem \ref{unifinvthm} gives a 
complete characterization of  evolutionary 
stability against multiple mutations in $2 \times 2$ games.  We omit the proof
as it is essentially contained in the proof of the Theorem \ref{unifinvthm}.

\begin{theorem}
For two player games, a  pure strategy $p$ is  evolutionarily stabile against multiple 
mutations if and  only if it is ESS.
\end{theorem}

\begin{remark}
We believe that this result is not true  for games with three or more strategies.
However, we neither have a proof nor a counter example.
\end{remark}

\section{Conclusions}
In this article, we introduced and studied the evolutionary stability against 
multiple mutations. We showed that the number of mutations  ( $m \geq 2$) 
is invariant. Further the evolutionarily stable strategy against multiple
mutations is necessarily a pure strategy. This notion coincides with ESS in the 
case of $2 \times 2$ symmetric games, as long as the ESS is pure. 
Like in the case of ESS, we do not have any general result on the existence.
Again it is in general non-unique.  In deed, strict Nash equilibrium, itself, can 
be non-unique.  

Our study also leaves several question to explore further. Firstly, note
that classical Hawk-Dove game does not have any evolutionarily stable 
strategy against multiple mutations. We do not know if this has any implication
in evolutionary biology as of now. 

In $2 \times 2$ case, an ESS is evolutionarily stable against multiple mutations
if and only if the ESS is pure. We believe that this result is not true in general 
case. However, we do not have any counter example.

Whether evolutionary stability against multiple mutations can be seen 
as a concept related to multiplayer games seems to be an interesting
issue to be explored. If such a connection can be drawn, we can, hopefully, 
apply our results in situations modeled as multiplayer games e.g.,  in bird nesting.

\begin{acknowledgements}
The initial impetus for the work reported here is provided while A.J. Shaiju was a
PostDoc at Laboratoire I3S, UNSA-CNRS, Sophia Antipolis, France. He profusely
thanks Prof.\ P.\ Bernhard for valuable and stimulating discussions. 
The authors record their sincere thanks to Prof. V.S. Borkar for various comments 
that helped improve the presentation of the article.
\end{acknowledgements}

\bibliographystyle{elsarticle-harv}
\bibliography{egt}

\end{document}